\documentclass[11 pt,a4paper,twoside,reqno]{amsart} 

 \usepackage{times, url}

\usepackage[english]{babel}

\usepackage{amsmath,amssymb,amsthm,amsfonts}

\newtheorem{theorem}{Theorem}[section]
\newtheorem{corollary}{Corollary}[theorem]
\newtheorem{lemma}[theorem]{Lemma}
\newtheorem{prop}[theorem]{Proposition}
\newtheorem{defin}[theorem]{Definition}
\newtheorem{rem}[theorem]{Remark}

\usepackage{graphicx,tikz,pgf}
\usetikzlibrary{calc,trees,positioning,arrows,chains,shapes.geometric,%
    decorations.pathreplacing,decorations.pathmorphing,shapes,%
    matrix,shapes.symbols}

\tikzset{
>=stealth',
  surveychain/.style={
    rectangle,
    rounded corners,
    draw=black, very thick,
    text width=17em,
    minimum height=2em,
    on chain},
  gapchain/.style={
    rectangle,
    rounded corners,
    text width=17em,
    minimum height=3em,
    text centered,
    on chain},
  decoration={brace},
}

\tikzset{
   snode/.style = {     draw,
                                rectangle,
                                minimum height  = 0.75cm,
                                minimum width  = 0.75cm,
                                node distance = -0.2mm  },
   gnode/.style = {     draw,
                                rectangle,
                                minimum height  = 0.75cm,
                                minimum width  = 3cm,
                                node distance = -0.2mm  },
   point/.style         =       {coordinate},
punkt/.style={
           rectangle,
           rounded corners,
           draw=black, very thick,
           text width=6.5em,
           minimum height=2em,
           text centered},
   pil/.style={
           ->,
           thick,
           shorten <=5pt,
           shorten >=5pt}
}

\usepackage{pgfplots}
\usepackage{filecontents}
\usepackage{dcolumn}
\newcolumntype{d}[1]{D{.}{.}{#1}}

\begin{document}

\begin{center}
\title
{Revisiting Generalized Bertand's Postulate and Prime Gaps}
\maketitle

{Madhuparna Das$^1$, Goutam Paul$^2$}\\
\vspace{3mm}
$^1$
e-mail: \url{amimadhuparna@gmail.com}\\
\vspace{2mm}
$^2$Cryptology and Security Research Unit,\\
                R. C. Bose Centre for Cryptology and Security,\\
                Indian Statistical Institute, Kolkata 700108, India.\\
e-mail: \url{goutam.paul@isical.ac.in}
\vspace{2mm}

\end{center}
\vspace{10mm}

\noindent
{\bf Abstract:} It is a well-known fact that for any natural number $n$, there always exists a prime in $[n, 2n]$. Our aim in this note is to generalize this result to $[n, kn]$. A lower as well as an upper bound on the number of primes in $[n, kn]$ were conjectured by Mitra et al. [Arxiv 2009]. In 2016, Christian Axler provided a proof of the lower bound which is valid only when $n$ is greater than a very large threshold. In this paper, after almost a decade, we for the first time provide a direct proof of the lower bound that holds for all $n \geq 2$. Further, we show that the upper bound is a consequence of Firoozbakht's conjecture. Finally, we also prove a stronger version of the bounded gaps between primes.
\vspace{5mm}

\section{Introduction}
A branch of number theory studying distribution laws of prime numbers among natural numbers. The Prime Number Theorem~\cite{pnt} describes the asymptotic distribution of prime numbers. It gives us a general view of how primes are distributed amongst positive integers and also states that the primes become less common as they become larger.

In 1845, Joseph Bertrand postulated~\cite{bp} that, there is always a prime between $n$ and $2n$, and he verified this for $n < 3 \times 106$. Tchebychev gave an analytic proof of the postulate in 1850, and a short but advanced proof was given by Srinivasa Ramanujan~\cite{ram}. An elementary proof is due to Paul Erd{\"o}s~\cite{erdos}, and totally there are eighteen different proofs of Bertrand Postulate.

In 2009, Mitra, Paul, and Sarkar~\cite{mps} had generalized the Bertrand's postulate to conjecture that for any integers $n$, $a$ and $k$, where $a=\lceil 1.1 \ln (2.5k) \rceil$, there are at least $k-1$ primes between $n$ and $kn$ when
$n \geq a$. In 2016, Christian Axler provided a proof~\cite{ca} of this conjecture using an application of generalized Ramanujan primes. However, his proof works for sufficiently large positive integers. In this paper, for the first time, we provide a direct proof (Theorem~\ref{thm1}) of the generalized Bertrand's postulate for all positive integers $n \geq 2$. {\bf This is the main contribution of our work}.

By the Bertrand Postulate, we know about the minimum number of primes in certain intervals. Now, there is an important question: what about the upper bound on the number of primes in similar interesting intervals. In the same 2009 paper, Mitra, Paul, and Sarkar~\cite{mps} had also conjectured that the upper bound of the number of primes in $[n,kn]$ is $\frac{kn}{n}+k^2$. In this paper, we provide a proof this conjecture. However, for this proof, we need a lower bound of the prime gap function, i.e., the gap $g_n$ between consecutive primes $p_n$ and $p_{n+1}$. Prime number theorem formalizes the intuitive idea that primes become less common as they become larger by precisely quantifying the rate at which this occurs. If we notice on the prime gap function, the first, smallest, and the only odd prime gap is the gap of size 1 between 2, the only even prime number, and 3, the first odd prime. There is only one pair of consecutive gaps having length 2: the gaps $g_2$ and $g_3$ between the primes 3, 5, and 7. All other prime gaps are even. For any integer $n$, the factorial $n!$ is the product of all positive integers up to and including $n$. Then in the sequence
\begin{equation*}
 n! + 2 , n ! + 3 , \ldots, n ! + n
\end{equation*}
the first term is divisible by 2, the second term is divisible by 3, and so on. Thus, this is a sequence of $n-1$ consecutive composite integers, and it must belong to a gap between primes having length at least $n-1$.

It follows that there are gaps between primes that are arbitrarily large, that is, for any integer $N$, there is an integer $m$ with $g_m \geq N$. As of March 2017, the largest known prime gap with identified probable prime gap ends has length 5103138, with 216849-digit probable primes found by Robert W. Smith~\cite{trn}. The largest known prime gap with identified proven primes as gap ends has length 1113106, with 18662-digit primes found by P. Cami, M. Jansen and J. K. Andersen~\cite{and}~\cite{primegap}.

We also study the upper as well as the lower bound of the prime gap function and show how can we apply it to prove the upper bound on the number of primes in the interval $[n, kn]$ for any positive integer $n$ and $k$ $\geq 1$ (Theorem~\ref{thm2}). Note that this result is dependent on the validity of Firoozbakht's conjecture~\cite{fir1,fir2}. However, this is not the main contribution of our work, as is Theorem~\ref{thm1}.

As a third contribution, we give a refinement for the bounded gaps between primes. This result is stronger than Bertrand Postulate and J Nagura's Theorem~\cite{nagura}. 
In this context, one should note that there exist some results that are better than Bertrand's Postulate, but they hold for sufficiently large values of $n$. In particular, here is a summary of them.
\begin{enumerate}
\item From the prime number theorem it follows that  for any real $\varepsilon> 0$ there is a $n_0 > 0$ such that for all $n > n_0$ there is a prime $p$ such $n < p < ( 1 + \varepsilon ) n$. It can be shown, for instance, that
\begin{align*}
\lim_{n\to\infty}  \frac{\pi((1+\varepsilon)n)-\pi(n)}{n / \log n} = \varepsilon.
\end{align*}
which implies that  $\pi( ( 1 + \varepsilon ) n )-\pi ( n )$ goes to infinity~\cite{res1}(and, in particular, is greater than 1 for sufficiently large $n$).
\item Non-asymptotic bounds have also been proved. In 1952, Jitsuro Nagura~\cite{res2} proved that for $n \geq 25$, there is always a prime between $n$ and $(1 + 1/5)n$. 
\item In 1976, Lowell Schoenfeld~\cite{res3} showed that for $n \geq 2010760$, there is always a prime between $n$ and $(1 + 1/16597)n$. 
\item In 1998, Pierre Dusart~\cite{res4} improved the result in his doctoral thesis, showing that for $k \geq 463$, $p_{k+1} \leq (1 + 1/(\ln^2p_k))p_k$, and in particular for $n \geq 3275$, there exists a prime number between $n$ and $(1 + 1/(2\ln^2n))n$.
\item In 2010, Pierre Dusart~\cite{res5} proved that for $n \geq 396738$ there is at least one prime between $n$ and $(1 + 1/(25 \ln^2n))n$. 
\item In 2016, Pierre Dusart~\cite{res6} improved his result from 2010, showing that for $n \geq 468991632$ there is at least one prime between $n$ and $(1 + 1/(5000 \ln^2n))n$. 
\item Baker, Harman and Pintz~\cite{res7} proved that there is a prime in the interval $\left[n, n+O\left(n^{21/40}\right)\right]$ for all large $n$. 
\end{enumerate}
Our work is a contribution on existence of a prime in the interval $(n, g(n))$, where $g(n) = n+\frac{n}{\lceil{1.1 \ln (2.5n)}\rceil}$ for all $n \geq 2$. Note that, in contrast to the interval $\left[n, n+O\left(n^{21/40}\right)\right]$, our result holds for $n \geq 2$ (Theorem~\ref{gapthm}).

In this connection, one may note that there are several recent works~\cite{rel1, rel2, rel3, rel4, rel5, rel6, rel7, rel8, rel9} on ``Bounded gaps between primes" for \textit{sufficiently large primes}. However, our work holds for all positive integers greater than equal to 2 and hence the above works are not comparable with ours.

\section{Preliminaries}
In this section we discuss some well-known facts, which play a critical role to prove our main result.
\begin{prop}\label{prop1}
\textbf{(Bertrand's Postulate):} For any integer $n > 3$  there always exists at least one prime number $p$ with, $n < p < 2 n$.
\end{prop}
On this notation we state some sets and functions.
\begin{defin}\label{def1}
\textbf{(Prime gap function):} A prime gap is the difference between two successive prime numbers. The $n$-th prime gap, denoted by $g_n$ or $g(p_n)$ is the difference between the $(n + 1)$-th and the $n$-th prime numbers, i.e.,
\begin{equation*}
g_n = p_{n+1} - p_n.
\end{equation*}
The function is neither multiplicative nor additive.
\end{defin}
\begin{defin}\label{def2}
\textbf{(Prime counting function):} The prime-counting function is the function counting the number of prime numbers less than or equal to some real number $x$. It is denoted by $\pi(x)$.
\end{defin}
\begin{prop}\label{prop2}
\textbf{(Prime Number Theorem):} The prime number theorem~\cite{pnt} states that $\frac{x}{\log x}$ is a good approximation to $\pi(x)$, in the sense that the limit of the quotient of the two functions $\pi(x)$ and $\frac{x}{\log x}$ as $x$ increases without bound is 1:
\begin{equation*}
\lim_{x\to\infty} \frac{\pi(x)}{\frac{x}{\log x}} = 1.
\end{equation*}
known as the asymptotic law of distribution of prime numbers. Using asymptotic notation this result can be restated as
\begin{equation*}
\pi(x) \sim \frac{x}{\ log x} .
\end{equation*}
\end{prop}
\begin{corollary}\label{cor1}
\textbf{(Bound for the $n$-th prime $p_n$):} The $n$-th prime number $p_n$ satisfies, $p_n \sim n \log n$ and the bound for the $n$-th prime number $p_n$ is~\cite{bnd}\cite{bnd2},
\begin{align*}
&n \ln n + n(\ln \ln n -1) < p_n < n \ln n + n(\ln \ln n) 
\end{align*}
\begin{align}\label{eq1}
&\implies n\left( \ln \frac{n \ln n}{e} \right) < p_n < n \ln (n \ln n).
\end{align}
\end{corollary}
Now, we will state some important results of number theory which will help us to prove our main result.

{\bf Firoozbakht’s conjecture:} The conjecture~\cite{fir1,fir2} states that ${p_n}^{1/n}$ (where $p_n$ is the $n$th prime) is a strictly decreasing function of $n$, i.e.,
\begin{align*}
\sqrt[n+1]{p_{n+1}} < \sqrt[n]{p_n} \text{ for all } n\geq1.
\end{align*}
Equivalently:
\begin{align*}
p_{n+1} < {p_n}^{1+\frac{1}{n}} \text{ for all } n\geq1.
\end{align*}
Farideh Firoozbakht verified her conjecture up to $4.444×10^{12}$. Alexei Kourbatov~\cite{fir3} has verified the same for all primes below $10^{19}$.

\begin{corollary}\label{prop3} \textbf{(Upper bound of the Prime gap function):}
If Firoozbakht conjecture is true then the prime gap function $g_n$ would satisfy~\cite{fir4}
\begin{align*}
 g_n < ( \log ⁡ p_n )^2 - \log ⁡ p_n  \text{ for all } n > 4.
\end{align*}
\end{corollary}

\begin{prop}\label{prop4}
Let $p_n$ be the $n$-th prime where $n \geq 1$. Then the following inequality is true:
\begin{align*}
n+1 \leq p_n.
\end{align*}
\end{prop}
\begin{proof}
We proof this inequality, by the method of Mathematical Induction. 
For $n = 1$, $p_1 = 2$, then equality holds. For, $n < 4$ this inequality holds trivially.
Now, let us assume that this inequality is true for $n-1$. So, we can write,
\begin{align*}
(n-1)+1 = n < p_{n-1}, \forall n \geq 4.
\end{align*}
 Now, for the $n$-th prime $p_n$,
\begin{align*}
&n < n+1 \leq p_{n-1} < p_n, \forall n \geq 4 \\
\implies & n+1 \leq p_n.
\end{align*} 
So, by the Induction Hypothesis this inequality holds. 
\end{proof}
\begin{prop}\label{prop5}
\textbf{ (Lower Bound of the Prime gap function):} Corollary~\ref{prop3} implies a strong form of Cramer's conjecture~\cite{crcn}\cite{crcn1}~\cite{crcn2} but is inconsistent with the heuristics of Granville and Pintz which suggests that,
\begin{equation*}
g_n > \frac{2- \varepsilon}{e^\gamma}(\log p_n)^2.
\end{equation*}
infinitely often for any $\varepsilon > 0$, where $\gamma$ denotes the Euler Mascheroni constant.
\end{prop}

\begin{defin}\label{def5}
\textbf{ (Primorial):} For the nth prime number $p_n$, the primorial $p_n\#$ is defined as the product of the first $n$ primes
\begin{align*}
p_n \# = \prod_{k=1}^{n} p_k,
\end{align*}
where $p_k$ is the $k$-th prime number.
\end{defin}
In the classic proof of $Bertrand$ $Postulate$ by Paul Erd{\"o}s~\cite{erdos}, 
it is shown that,
\begin{align*}\label{eq8}
x\#<4^x,
\end{align*}
where $x\#$ is the primorial for $x$.
\begin{prop}\label{prop6}
For all $n \geq 2$, $n\# = {\prod_{p \leq n}} p \leq 4^n$, where the product is over primes.
\end{prop}
\begin{proof}
We proceed by induction on $n$. For small values of $n$, the claim is easily verified. For larger even $n$ we have
\begin{align*}
 {\prod_{p \leq n}} p =  {\prod_{p \leq n}} p \leq 4^{n-1} \leq 4^n,
\end{align*}
the equality following from the fact that, $n$ is even and so not a prime, and the first inequality following from the inductive hypothesis. For larger odd $n$ say $n = 2m+1$, we have
\begin{align*}
&{\prod_{p \leq n}} p =  {\prod_{p \leq m+1}} p  {\prod_{m+2 \leq p \leq 2m+1}} p \\
\leq& 4^{m+1} \binom{2m+1}{m}\\
\leq&4^{m+1}2^{2m} = 4^{2m+1} = 4^n. 
\end{align*}
\end{proof}

\section{Generalization of Bertrand's Postulate}
In this section our aim is to prove the Generalization of Bertrand's Postulate. Before proving that, we discuss some intermediate results that will lead to our first main result.
\begin{defin}
For the positive integer $a$ and $b$ we denote the set of integers $a, a+1, \ldots, b$ by $[a, b]$.
\end{defin}
\begin{defin}\label{def6}
Define the set $S_i$, as
\begin{equation}\label{eq2}
S_i = \{k : f(k) = i+1\},
\end{equation}
where $f(k) = \lceil 1.1 \ln(2.5k) \rceil$ and $k \in \mathbb{N}$.
\end{defin}
\begin{prop}
For all values of $k \geq 2$,
\begin{equation*}
f(k+1) = \left\{ \begin{array}{ll}
   either f(k) & \\or,f(k)+1; \\
  \end{array} \right.
\end{equation*}
\end{prop}
\begin{proof} Definition~\ref{def6} makes this result obvious. \end{proof}
\begin{lemma}\label{lem1} Let, $f(k) = \lceil 1.1 \ln(2.5k) \rceil$, then the inequality,
 \begin{equation}\label{eq3}
\left|{\left(\log{p_{\left(f(k)+k-3\right)}}\right)}^2 - \left(\log{p_{(f(k)+k-3)}}\right)\right| \ < \left(k+1)(f(k)+1\right) - p_{\left(f(k)+k-3\right)}.
\end{equation}
holds for all $k \geq 5$.
\end{lemma}
\begin{proof}
Let us consider that $n$ = $f(k)+k-3$, $\forall k \geq 5$. Now, $k$ is an integer, so we can reduce the inequality in terms of $k$ and $n$, which is,
\begin{equation}\label{eq4}
\left|(\log p_n)^2 - (\log p_n)\right| <  (k+1)(n+3-k) - p_n.
\end{equation}
We will use the bound for the $n$-th prime $p_n$ as stated in Corollary~\ref{cor1}. So from inequality~\eqref{eq1}, the inequality~\eqref{eq4} become,
\begin{align*}
&\left|(\log p_n)^2 - (\log p_n)\right| \\
<& \left|\left(\log {(n \ln (n \ln n)}\right)^2 - \left(\log n  \left(\ln\frac{n \ln n}{e}\right)\right)\right| \\
<& (k+1)(n+3-k) - n \left(\ln\frac{n \ln n}{e}\right) \\
<& (k+1)(n+3-k) - p_n.
\end{align*}
If we can prove the inequality,
\begin{equation}\label{eq5}
 \left|\left(\log {(n \ln (n \ln n)}\right)^2 - \left(\log n  \left(\ln\frac{n \ln n}{e}\right)\right)\right| < (k+1)(n+3-k) - n \left(\ln\frac{n \ln n}{e}\right), \forall k \geq 5.
\end{equation}
then we are done.\\
By our assumption, $n = f(k)+k-3$, $n \geq k$.\\
Therefore, $n \ln (n \ln n) < n + 3 - k$ holds, $\forall n$ and $k$. Now, 
\begin{align*}
&\left(\log {\left(n \ln (n \ln n)\right)}\right)^2
<\log {(n \ln (n \ln n)}) (k+1) \\
&<(n + 3 - k) (k+1).
\end{align*}
Again, we have,$\big(\log ( n \ln (n \ln n)\big) < n \left(\ln\frac{n \ln n}{e}\right)$.
From this we can conclude the inequality~\eqref{eq5}. 
\end{proof}
\begin{lemma}\label{lem2}
Let, $f(k) = \lceil 1.1 \ln(2.5k) \rceil $, then the inequality,
\begin{equation}\label{eq6}
\left|{\left(\log{p_{(f(k)+k +r)}}\right)}^2 - \left(\log{p_{(f(k)+k +r)}}\right)\right| \ < \ (k+4+r)(f(k)+1) - p_{(f(k)+k +r)}.
\end{equation}
holds for all $k \geq 5$, where $r \in [-2, \infty) \in \mathbb{Z}$.
\end{lemma}
\begin{proof}
We will prove this inequality, by induction on $r$.

For the base case, we have to show that for $r = -2$, this inequality,
\begin{equation}\label{eq7}
\left|{\left(\log{p_{(f(k)+k -2)}}\right)}^2 - \left(\log{p_{(f(k)+k -2)}}\right)\right| \ < \ (k+2)(f(k)+1) - p_{(f(k)+k -2)}.
\end{equation}
holds.
As we proof the inequality in Lemma~\ref{lem1}, we can prove this by the similar argument, where $n = f(k)+k-2$.

Let us assume that the inequality~\eqref{eq7} is true for $r$, where $n = f(k)+k+r$.
Now, for $r+1$, we have $n+1 = f(k)+k+r+1$. The inequality becomes,
\begin{align*}
&\left|\left(\log {((n+1) \ln ((n+1) \ln (n+1))}\right)^2 - \left(\log (n+1) \ln \frac{(n+1) \ln (n+1)}{e}\right)\right|\\ 
<&(k+4+r+1)(n-k-r) - (n+1) (\ln\frac{(n+1) \ln (n+1)}{e}).
\end{align*}
As we proved the inequality~\eqref{eq3}, in Lemma~\ref{lem1}, we can prove this by the similar argument.
So, by the Induction Hypothesis, the inequality~\eqref{eq6} holds. 
\end{proof}
Now we are in a position to prove our first main result.
\begin{theorem}\label{thm1}
For any integers $n$ and $k$, where $f(k) = \lceil 1.1 \ln(2.5k) \rceil$, then there are at least $(k-1)$ primes between $n$ and $kn$, when $n \geq f(k)$.
\end{theorem}
\begin{proof}
The proof is by induction on the number of elements $k \in S_i$ and on $i$, index of the set $S_i$. We divide this proof into two parts.

\textbf{{First part:}} In this part we prove that this theorem is true for all $k \in S_i$.
To prove this, take any set $S_i$, $\forall i = 1, 2, \ldots$.\\
For the base step of the induction, consider $min\{S_i\} = k_{min}$.
We have to prove that there are at least $\left(min\{S_i\}\right)-1$ primes between the gap, 
\begin{align*}
\left[f\left(min\{S_i\}\right), \left(min\{S_i\} \times f\left(min\{S_i\}\right)\right)\right].
\end{align*}
For this gap we start our prime counting from $p_i$, and clearly,
\begin{align*}
 p_i \in \left[f\left(min\{S_i\}\right), \left(min\{S_i\} \times f\left(min\{S_i\}\right)\right) \right].
\end{align*}
Now, we have to prove that, 
\begin{align*}
p_{i+(min\{S_i\})-2} \in \left[f\left(min\{S_i\}\right), \left(min\{S_i\} \times f\left(min\{S_i\}\right)\right) \right].
\end{align*}
Since, we have $i = f(min\{S_i\})-1$, so now if we able prove that, 
\begin{align*}
p_{(f(min\{S_i\})+min\{S_i\}-3)} \in \left[f\left(min\{S_i\}\right), \left(min\{S_i\} \times f\left(min\{S_i\}\right)\right)\right],
\end{align*}
then we are done. Thanks to Lemma~\ref{lem2}, this result holds.
Now we start our inductive argument to prove that if this theorem is true for $min\{S_i\}$, then it is also true for all elements of $\{S_i\}$. Let us assume that this theorem is true for an element $k \in S_i$ and $k \ne max\{S_i\}$. Since, by our assumption we can say that there exist $k-1$ primes between the gap 
\begin{align*}
\left[f(k), \left(k\times f(k)\right)\right].
\end{align*}
By the definition of the set $S_i$, $f(k)$ remains same for all $k \in S_i$. To prove, the $first$ $part$ of this theorem we have to prove that there exist $k$ primes between the gap 
\begin{align*}
\left[f(k+1), \left(k+1\times f(k+1)\right)\right].
\end{align*}
So, we can start our prime counting from $p_i$. We have to prove that, there exists at least one prime between the gap, 
\begin{align*}
\left[k\times f(k), k+1 \times f(k)\right].
\end{align*}
By the Lemma~\ref{lem2}, which is true.

So, Induction shows that this theorem is true for all $k \in S_i$.

\textbf{Second part:} In the $second$ $part$ of the proof we prove that this theorem is true for all such set $S_i, \forall i = 1, 2, \ldots$. We proceed by induction on $i$.

For the base step of the induction we have $S_1 = \{2\}$ and $f(2) = 2$. Clearly, $2 < 3 < 4$, here 3 is the required prime. 

Now,  we start our prime counting for the set $S_i$, from, $p_i$, where $p_i$ is the $i$-th prime. Clearly, 
\begin{align*}
p_i \in \left[f(k), kf(k)\right].
\end{align*}

Let us assume that this theorem is true for the set $S_{i-1}, \forall i \geq 1$. Then by the $first$ $part$ of the proof we can say that for all elements of $S_{i-1}$ this theorem holds.
By the $first$ $part$ of the proof this theorem holds for $max\{S_{i-1}\}$. Now, we show that this also holds for $min\{S_i\}$.

For the set $S_i$, 
\begin{align*}
f(min\{S_i\}) = i+1 \implies i = f(min\{S_i\}) -1.
\end{align*}
We have started, the counting of primes for the set $S_i$ from $p_i$, where $p_i$ is the $i$-th prime.\\
Then for the set $S_{i-1}$ we count $max\{S_{i-1}\}-1$ primes from $p_{i-1}$. So, clearly between the gap 
\begin{align*}
\left[f(min\{S_i\}) = f(max\{S_{i-1}\})+1, max\{S_{i-1}\} \times f(max\{S_{i-1}\})\right],
\end{align*}
there are at least $max\{S_{i-1}\}-2$ primes.

Now, we have to show that there are at least 2-primes between the gap,
\begin{align*}
\left[max\{S_{i-1}\}f(max\{S_{i-1}\}), (max\{S_{i-1}\}+1)f(max\{S_{i-1}\}+1)\right].
\end{align*}
We have 
\begin{align*}
max\{S_{i-1}\}+1 = min\{S_i\}
\end{align*}
and 
\begin{align*}
f(max\{S_{i-1}\})+1 = f(min\{S_i\}).
\end{align*}
By the prime gap function $g_n$, and Lemma~\ref{lem1}, we have proved that there are at least 2-primes between the gap,
\begin{align*}
\left[max\{S_{i-1}\}f(max\{S_{i-1}\}), (max\{S_{i-1}\}+1)f(max\{S_{i-1}\}+1)\right].
\end{align*}
Induction shows that this conjecture is true for all set $S_i$.
From the $first$ and $second$ part of this proof we can say that this theorem is true for all elements of the set $S_i$, and it is also true for all such set $S_i, \forall i = 1, 2, \ldots$, which implies that there are at least $(k-1)$ primes between $n$ and $kn$, when $n \geq f(k)$, where $f(k) = \lceil 1.1 \ln(2.5k)\rceil$ and $k \in \mathbb{Z}$. 

We have proved the $Generalization$ $of$ $Bertrand's$ $Postulate$. 
\end{proof}

\section{Upper bound on the number of primes}
In this section, we prove the upper bound on the number of primes for the interval $[n,kn]$, where $n, k$ are positive integers greater than equal to 2. To prove our main result of this section, first we state and prove the following lemma.
\begin{lemma}\label{lem3}
Let $n$ be a non-zero positive integer $\geq 5$. Then the inequality
\begin{align*}
\left(\frac{2n}{9}+4\right) \times \left(\log n \ln (n \ln n)\right)^2 > n
\end{align*}
holds.
\end{lemma}
\begin{proof}
The proof is by induction on $n$. 

For the base step of the induction consider $n = 5$, we can write,
\begin{align*}
&\left(\frac{2 \times 5}{9}+4\right) \times \left(\log 5 \ln (5 \ln 5)\right)^2 \\
&=  5.111 \times 0.018 = 5.202 > 5.
\end{align*}
So, the base case holds.\\
Now, let us assume that this inequality is true for some $n$. We have,
\begin{align*}
&\left(\frac{2n}{9}+4\right) \times \left(\log n \ln (n \ln n)\right)^2 > n
\end{align*}

Now,

\begin{align*}
&\left(\frac{2n}{9}+4\right) \times \left(\log (n+1) \ln ((n+1) \ln (n+1)\right)^2 +\\
&\frac{2}{9} \times  \left(\log (n+1) \ln ((n+1) \ln (n+1)\right)^2\\
>&n+  \frac{2}{9} \times  \left(\log (n+1) \ln ((n+1) \ln (n+1)\right)^2\\
>&n+1.
\end{align*}
This implies, the inequality holds for $(n+1)$ as well. 
\end{proof}

Now we prove our second main result~\cite{mps}.
\begin{theorem}\label{thm2}
Given a positive integer $k$, then the number of primes between $n$ and $kn$, for any positive integer $n$, is bounded by $\frac{kn}{9}+k^2$.
\end{theorem}
\begin{proof}
The proof is by induction on $k$.

For the base step of the induction consider $k = 2$, the upper bound on the number of primes between the gap $[n, 2n]$ is $\frac{2n}{9}+4$.
To prove this we use the lower bound of the prime gap function $g_n$, given by Proposition~\ref{prop5}, as Corollary~\ref{prop3} implies a strong form of Cramer's conjecture~\cite{crcn}\cite{crcn1}\cite{crcn2}, but is inconsistent with the heuristics of Granville and Pintz which suggests the following:
\begin{equation*}
g_n > \frac{2- \varepsilon}{e^\gamma}(\log p_n)^2.
\end{equation*}
infinitely often for any $\varepsilon > 0$, where $\gamma$ denotes the Euler Mascheroni constant.
We can approximate the part 
\begin{align*}
\frac{2-\varepsilon}{e^\gamma}.
\end{align*}
Since, $g_n$ is a prime gap function, so it can not be negative. The maximum possible value of $\varepsilon$ is 2.

As we know the value of 
\begin{align*}
&e \approx 2.718 \ldots, \text{ and } \gamma \approx 1.781062\ldots\\
\implies & \frac{2-\varepsilon}{e^\gamma} \approx 0.561 \ldots \approx 1.
\end{align*}
We are using the counting method to prove this upper bound. If we take the minimum gaps between primes that is the lower bound of prime gap function, then we can count the maximum number of primes for any interval. 
To do so we have to prove the inequality,
\begin{align*}
\frac{2n}{9}+4 \times (\log p_n)^2 > 2n - n = n.
\end{align*}
We will use the upper bound of the $n$-th prime $p_n$, which we have stated in the section of known result. So, the inequality becomes,
\begin{align*}
\frac{2n}{9}+4 \times \left(\log n \ln (n \ln n)\right)^2 > 2n - n = n.
\end{align*}

Thanks to  Lemma~\ref{lem3} base case holds.  we can now modify the proof of this theorem inductively. Now let us assume that the inequality,
\begin{align*}
\left(\frac{kn}{9}+k^2\right) \times \left(\log (n \ln (n \ln n)) \right)^2 > kn - n = (k-1)n.
 \end{align*}
is true. So, we can write,
\begin{align*}
&\left(\frac{kn}{9}+k^2\right) \times  \left(\log (n \ln (n \ln n)) \right)^2 + \left(\frac{n}{9} + 2k + 1\right) \times  \left(\log (n \ln (n \ln n)) \right)^2\\
=&\left(\frac{(k+1)n}{9}+(k+1)^2\right) \times  \left(\log (n \ln (n \ln n)) \right)^2 \\
>& \left((k-1)n\right) +  \left(\frac{n}{9} + 2k + 1\right) \times  \left(\log (n \ln (n \ln n)) \right)^2\\
>& (k-1)n + n\\
=& kn.
\end{align*} 
Induction shows that for a given  positive integer $k$, the number of primes between $n$ and $kn$, for any positive integer $n$, is bounded by $\frac{kn}{9}+k^2$. 
\end{proof}

\section{Gaps between Primes}
In this section, we construct the interval $[n, g(n)]$, for $g(n) = n+\frac{n}{\lceil{1.1 \ln (2.5n)}\rceil}$, containing at least one prime. This is our third main result and it is stronger than Bertrand's Postulate and J. Nagura's Theorem. 

\begin{theorem}\label{gapthm}
There exist at least one prime number $p$ with,\\ $kf(k) <p < k(f(k)+1)$, where $f(k) =  \lceil 1.1 \ln(2.5k)\rceil$ and $k \in \mathbb{Z}^+$.
\end{theorem}

\begin{proof}
The proof is by induction on $k$. We prove this theorem into two parts. For the $first$ $part$, we use inductive argument on $k$, for which $f(k)$ is same. In the $second$ $part$, we use inductive argument on $k$, for which $f(k+1) = f(k)+1$.

\textbf{ First Part:} In this part we prove that this theorem is true for all $k$, such that $f(k+1) = f(k)$.
For the base step of the induction consider, $k = 2$ and $f(2) = 2$. Trivially, we can write, $4 < 5 < 6$, where $p_3 = 5$.
To complete our inductive argument now let us assume that, there exist a prime $p_m$ between the gap
\begin{align*}
\left[kf(k), k(f(k)+1)\right],
\end{align*} 
for some integer $m$, and $p_m$ is the $m$-th prime. Then we have to prove that, there exist at least one prime between the gap,
\begin{align*}
\left[(k+1)f(k), (k+1)(f(k)+1)\right].
\end{align*} 
Trivially, we can write $m < kf(k)$. So, if we can prove that,
\begin{align*}
p_{m+r} \in \left[(k+1)f(k), (k+1)(f(k)+1)\right],
\end{align*} 
for some integer $r \geq 1$ then we are done. By our assumption,
\begin{equation*}
 kf(k) < p_m < kf(k)+k,
\end{equation*}
this inequality holds. Then we have to prove that there exist at least one prime between the gap,
\begin{align*}
\left[(k+1)f(k), (k+1)(f(k)+1)\right].
\end{align*}
There are two possible cases to consider.

\textbf{Case (i):} If $p_m > (k+1)f(k)$, then there is nothing to prove.

\medskip

\textbf{Case (ii):} If $p_m < (k+1)f(k)$, then we have to show that
\begin{align*}
p_{m+r} < (k+1)(f(k+1)),
\end{align*}
for some integer $r \geq 1$. So we can write,
\begin{align*}
&p_{m+r} = p_m+\sum_{i=1}^{r} g_{m+i} < (k+1)f(k)+\sum_{i=1}^{r} \left(\left(\log p_{m+i}\right)^2-\log p_{m+i}\right)\\
\implies&p_m+\sum_{i=1}^{r} g_{m+i} < (k+1)f(k)+\sum_{i=1}^{r} \left(\log p_{m+i}\right)^2-\sum_{i=1}^{r} \left(\log p_{m+i}\right)\\
=&(k+1)f(k)+\left(\log p_{m+1}+ \ldots +\log p_{m+r}\right)^2-\left(\log p_{m+1}+ \ldots +\log p_{m+r}\right)\\
=&(k+1)f(k)+\left(\log {\frac{p_{m+r}\#}{p_m\#}}\right)^2-\left(\log {\frac{p_{m+r}\#}{p_m\#}}\right).\\
\end{align*}
From the proposition~\ref{prop6}, we can write,
\begin{align*}
&(k+1)f(k)+\left(\log {\frac{p_{m+r}\#}{p_m\#}}\right)^2-\left(\log {\frac{p_{m+r}\#}{p_m\#}}\right)\\
<&(k+1)f(k)+\left(\log {\frac{4^{p_{m+r}}}{4^{p_m}}}\right)^2-\left(\log {\frac{4^{p_{m+r}}}{4^{p_m}}}\right)\\
<&(k+1)f(k)+\left((p_{m+r}-p_m) \log 4\right)^2-\left((p_{m+r}-p_m) \log 4\right)\\
=&(k+1)f(k)+\left((p_{m+r}-p_m) \times 0.6020\right)^2-\left((p_{m+r}-p_m) \times 0.6020\right)\\
<&(k+1)f(k)+\left((p_{m+r}-(k+1)f(k)) \times 0.6020\right)^2-\left((p_{m+r}-(k+1)f(k)) \times 0.6020\right)\\
<&(k+1)f(k)+(k+1) = (k+1)(f(k)+1).
\end{align*}
So, Induction shows that there exists a prime between the gap
\begin{align*}
\left[(k+1)f(k), (k+1)(f(k)+1)\right].
\end{align*}
\textbf{ Second Part:} In this part we prove that this theorem is true $k$, such that $f(k+1) = f(k)+1$.

For the base step of the induction, $k = 2$ and $f(2) = 2$. The desired inequality holds trivially.

Now we start our inductive argument. Let us assume that there exists a prime $p_m$ between the gap $kf(k)$ and $k(f(k)+1)$, for some integer $m$, and $p_m$ (the $m$-th prime). Then we have to show that there exist one prime between the gap,
\begin{align*}
\left[(k+1)(f(k)+1), (k+1)(f(k)+2)\right].
\end{align*}
So, if we can prove that
\begin{align*}
p_{m+r} \in \left[(k+1)(f(k)+1), (k+1)(f(k)+2)\right],
\end{align*}
for some integer $r \geq 1$, then we are done. By our assumption,
\begin{equation*}
 kf(k) < p_m < kf(k)+k,
\end{equation*}
holds. Then we have to prove that there exist at least one prime between the gap
\begin{align*}
\left[(k+1)(f(k)+1), (k+1)(f(k)+2)\right].
\end{align*}
There are two possible cases.

\textbf{ Case (i):} If $p_m > (k+1)(f(k)+1)$, then there is nothing to prove.

\medskip

\textbf{ Case (ii):} If $p_m <(k+1)(f(k)+1)$, then we have to show that,
\begin{align*}
p_{m+r} < (k+1)(f(k+2)),
\end{align*}
for some integer $r \geq 1$. So, we can write,
\begin{align*}
&p_{m+r} = p_m+\sum_{i=1}^{r} g_{m+i} < (k+1)(f(k)+1)+\sum_{i=1}^{r} \left(\left(\log p_{m+i}\right)^2-\log p_{m+i}\right)\\
\implies&p_m+\sum_{i=1}^{r} g_{m+i} < (k+1)(f(k)+1)+\sum_{i=1}^{r} \left(\log p_{m+i}\right)^2-\sum_{i=1}^{r} \left(\log p_{m+i}\right)\\
=&(k+1)(f(k)+1)+\left(\log p_{m+1}+ \ldots +\log p_{m+r}\right)^2-\left(\log p_{m+1}+ \ldots +\log p_{m+r}\right)\\
=&(k+1)(f(k)+1)+\left(\log {\frac{p_{m+r}\#}{p_m\#}}\right)^2-\left(\log {\frac{p_{m+r}\#}{p_m\#}}\right).\\
\end{align*}
By the Proposition~\ref{prop6}, we can write,
\begin{align*}
&(k+1)(f(k)+1)+\left(\log {\frac{p_{m+r}\#}{p_m\#}}\right)^2-\left(\log {\frac{p_{m+r}\#}{p_m\#}}\right)\\
<&(k+1)(f(k)+1)+\left(\log {\frac{4^{p_{m+r}}}{4^{p_m}}}\right)^2-\left(\log {\frac{4^{p_{m+r}}}{4^{p_m}}}\right)\\
<&(k+1)(f(k)+1)+\left((p_{m+r}-p_m) \log 4\right)^2-\left((p_{m+r}-p_m) \log 4\right)\\
=&(k+1)(f(k)+1)+\left((p_{m+r}-p_m) \times 0.6020\right)^2-\left((p_{m+r}-p_m) \times 0.6020\right)\\
<&(k+1)(f(k)+1)+\left((p_{m+r}-(k+1)(f(k)+1)) \times 0.6020\right)^2\\
&-\left((p_{m+r}-(k+1)(f(k)+1)) \times 0.6020\right)\\
<&(k+1)(f(k)+1)+(k+1) = (k+1)(f(k)+2).
\end{align*}
Hence, by the Induction Hypothesis we have proved that there exists a prime between the gap,
\begin{align*}
\left[(k+1)(f(k)+1),(k+1)(f(k)+2))\right].
\end{align*}
By the $first$ and $second$ part of the proof and Induction Hypothesis theorem~\ref{gapthm} is true. 
\end{proof}

\begin{rem}
The statement in the Theorem~\ref{gapthm}, states about the gap, between which there exist at least one prime. This gap is smaller than the gap which stated in Bertrand's Postulate and J Nagura's Theorem for all $k \geq 38$. To see this, let
\begin{align*}
&n = k \times f(k), \mbox{where} f(k) = \lceil 1.1\ln(2.5k) \rceil
\implies n \geq k,
\end{align*}
because for all $k \geq 2$, $f(k) \geq 2$. In Bertrand's Postulate, the length of the gap is $n$ and in our result the length of the gap is $k$. Similarly, we can prove that our gap is smaller than that of J Nagura's Theorem. 
\end{rem}
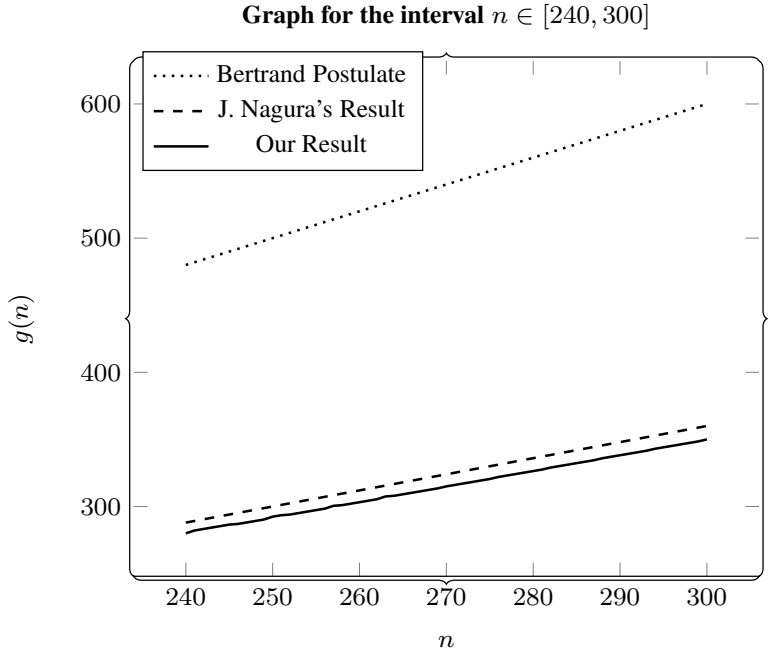
\begin{figure}[h]
\centering
\begin{minipage}[t]{1.0\textwidth}
\centering
\begin{tikzpicture}[scale=1.2,font=\scriptsize]

\begin{axis}[title={\textbf{Graph for the interval $n\in[240,300]$}},xlabel={n \in mathbb{N}},xlabel={$n$},ylabel={$g(n)$},legend style={at={(axis cs:235,550)},anchor=south west}]
\addplot[mark=none,dotted,thick, color=black] table[x index=0,y index=1,col sep=comma] {datax.dat};
\addlegendentry{Bertrand Postulate}

\addplot[mark=none,dashed,thick, color=black] table[x index=0,y index=2,col sep=comma] {datax.dat};
\addlegendentry{J. Nagura's Result}

\addplot[mark=none,solid,thick, color=black] table[x index=0,y index=3,col sep=comma] {datax.dat};
\addlegendentry{Our Result}

\end{axis}

\end{tikzpicture}
\caption{\scriptsize {Comparing the interval between which a prime must exist, according to Bertrand' Postulate, J. Nagura's theorem and Our Result}}
\label{f3}
\end{minipage}

\end{figure}

In this graph, we have compared the different intervals $(n, g(n))$, where 
$g(n) = 2n$ for Bertrand Postulate (dotted line),  $g(n) = \frac{6}{5}n$ for J. Nagura's Theorem (dashed line) and $g(n) = n+\frac{n}{\lceil 1.1\ln(2.5n)\rceil}$ for our Theorem~\ref{gapthm} (solid line). Note that, for $n = 240$, the value of $g(n)$ is 
480, 288 and 280 respectively for the three cases.

\section{Conclusion}
In this paper, we have proved Generalization of Bertrand's Postulate and upper bound on the number of primes in the interval $n$ and $kn$, for any positive integer $n$ and $k \geq 1$. To derive the lower bound, no unproven conjecture is required. However, to prove the upper bound, we have used the lower bound of the prime gap function, which comes as a consequence of Firoozbakht's conjecture. 
Finally, we have also been able to tighten the threshold beyond $n$, before which there must exit at least one prime.

\makeatletter
\renewcommand{\@biblabel}[1]{[#1]\hfill}
\makeatother

\end{document}